\documentclass[12pt,leqno,fleqn]{amsart}  
\usepackage{amsmath,amstext,amsthm,amssymb,amsxtra}
\usepackage[top=1.5in, bottom=1.5in, left=1.25in, right=1.25in]	{geometry}
\usepackage{txfonts} % pxfonts txfonts 
\usepackage[T1]{fontenc}
\usepackage{lmodern}

 \usepackage{euler}   % better than the option below

\usepackage{tikz}\usepackage{tikz-3dplot}

\usepackage{mathtools}
\mathtoolsset{showonlyrefs,showmanualtags}

\usepackage{hyperref} %,hypertexnames=false,colorlinks,[pagebackref]
\hypersetup{
    colorlinks=true,       % false: boxed links; true: colored links
    linkcolor=blue,          % color of internal links
    citecolor=magenta,        % color of links to bibliography
    filecolor=magenta,      % color of file links
    urlcolor=cyan           % color of external links
}

\usepackage[msc-links,backrefs]{amsrefs}

\theoremstyle{plain} % definition 
\newtheorem{lemma}[equation]{Lemma} 
\newtheorem{proposition}[equation]{Proposition} 
\newtheorem{theorem}[equation]{Theorem} 
\newtheorem{corollary}[equation]{Corollary} 

\newtheorem{priorResults}{Theorem}

\theoremstyle{definition}
\newtheorem{definition}[equation]{Definition}

\theoremstyle{remark}
\newtheorem{remark}[equation]{Remark}

\newtheorem*{ack}{Acknowledgment}

\numberwithin{equation}{section}

\title[Random Embeddings] {Random Tessellations, Restricted Isometric Embeddings, and One Bit Sensing}
\author{Dmitriy Bilyk}   %  can use \and  

\address{School of Mathematics, University of Minnesota, Minneapolis MN 55455, USA}
\email {}
\thanks{Research supported in part by}

\author{Michael T. Lacey}   %  can use \and  

\address{ School of Mathematics, Georgia Institute of Technology, Atlanta GA 30332, USA}
\email {lacey@math.gatech.edu}
\thanks{Research supported in part by grant NSF-DMS 1265570. }

\begin{document}

\begin{abstract}
We obtain mproved bounds for one bit sensing.  For instance, let $ K_s$ denote the set of  $ s$-sparse 
unit vectors in the sphere $ \mathbb S ^{n}$
in dimension $ n+1$ with sparsity parameter $ 0 < s < n+1$ and assume that $ 0 < \delta < 1$. We show that for    $ m \gtrsim \delta ^{-2} s \log \frac ns$,  the one-bit map   
\begin{equation*}
x \mapsto   \bigl[  \textup{sgn} \langle x,g_j \rangle \bigr] _{j=1} ^{m} ,
\end{equation*}
where $ g_j$ are iid gaussian vectors on $ \mathbb R ^{n+1}$,  with high probability has $ \delta $-RIP from  $ K_s$  into the $ m$-dimensional Hamming cube.   
These bounds match the bounds for the \emph{linear} $ \delta $-RIP  given by $ x \mapsto \frac 1m[\langle x,g_j \rangle ] _{j=1} ^{m} $, from the 
sparse vectors in $ \mathbb R ^{n}$ into $ \ell ^{1}$.  In other words,  the one bit and linear RIPs are equally effective. 
There are corresponding improvements for other one-bit properties, 
such as  the sign-product RIP property.  

\end{abstract}

	\maketitle  
\setcounter{tocdepth}{1}\tableofcontents

%%%%%%%%%%%%%%%%%%%%%%%%%%%%%% SECTION  SECTION SECTION 
\section{Introduction} %\label{s:}

We study questions associated with one bit sensing and dimension reduction. 
We shall  work with subsets $ K $ of the unit sphere $ \mathbb S^{n}$ in dimension $ n +1$.  A natural 
class of such subsets, which arises in compressed sensing, are the \emph{$ s$-sparse vectors} 
\begin{equation}\label{e:sparse}
K _{s} := \{  x = (x_1 ,\dotsc, x_{n+1}) \in \mathbb S ^{n} \;:\;   \sharp \{ j \;:\; x_j \neq 0\}\leq s \},
\end{equation}
 where $0<s<n+1$
 
We denote by  $ d (x,y)$ the  geodesic distance on $ \mathbb S ^{n}$,  normalized so that the distance between antipodal points is one,  that is 
\begin{equation}
d(x,y) = \frac{\cos^{-1} ( x\cdot y) }{\pi}.
\end{equation}

For  iid uniform samples $ \theta _ j$, $ 1\le j \le m$,  we will consider the  linear map 
\begin{equation} \label{e:A}
A x =  \frac 1 m \begin{bmatrix}
\langle  \theta _j  ,x  \rangle
\end{bmatrix} _{1\le j \leq m}
\end{equation}
from $ \mathbb R ^{n+1}$ to $ \mathbb R ^{m}$, where  $ \mathbb R ^{n+1}$ is equipped with  the usual Euclidean $ \ell ^2 $ metric, while  we use the $ \ell ^{1}$ metric
on $ \mathbb R ^{m}$.

More generally, we will consider the \emph{sign-linear map}
\begin{equation} \label{e:sgnA}
\textup{sgn} (Ax) := \begin{bmatrix} \textup{sgn} (\langle  \theta _j  ,x  \rangle)   \end{bmatrix} _{1\le j \leq m}
\end{equation}
which is a map from $\mathbb R^{n+1}$ into the Hamming cube $ \{ -1, 1\} ^{m}$   with the 
metric 
\begin{equation}\label{e:dhamming}
 d _{H} (x,y) = \frac 1 {2m} \sum_{i=1} ^{  m } \lvert  x_i - y_i\rvert  = \frac{\{ 1\le j \le m:\, x_j \neq y_j \} }{m} , 
 \end{equation} i.e. the fraction of coordinates 
in which $ x, y \in \{-1, +1\} ^{m}$ differ.  In other words, a sign-linear map keeps only one bit of information from each linear measurement.     This is an example of one-bit sensing, 
a subject initiated by Boufounos and Baraniuk \cite{1Bit}, and further studied by 
several authors \cites{13051786,MR3043783,MR3069959,MR3164174}. 

The motivation for the one-bit map is that it is a canonical non-linearity in measurement, 
and so its study opens the door for a broader non-linear theory in compressed sensing.  
It has proven to be useful in other settings, see the applications of one-bit sensing  in 
\cites {14078246,13051786}.  Furthermore, this topic has deep connections to 
the Dvoretsky Theorem, as explained in \cite{MR3164174}, as well as important relations to geometric
discrepancy theory, explored  in a companion paper of the authors \cite{BL1}.

The basic question,  with different meanings and interpretations being 
attached to  it,  is how well these maps preserve the structure of $ K$.  Generally, it is also of interest to study this question when the iid samples $ \theta _j$ are, for instance,
standard gaussians on $ \mathbb R ^{n+1}$.  However, in the context of the one-bit map, the magnitudes of both $ x$ and $ \theta _j$ 
are lost, so that we concern ourselves mostly with the $ \theta _j$ being uniform on the sphere $ \mathbb S ^{n}$. 

The first crude property begins with the observation that the map $ \textup{sgn} (\langle \theta _j , x \rangle)$
divides $ \mathbb S ^{n}$ into two hemispheres, bounded by the hyperplane $ \theta _j ^{\perp}$. 
Therefore, a number of one bit observations induce a tessellation of a subset $ K\subset \mathbb S ^{n}$ into cells 
bounded by the hyperplanes.  Two points are in the same cell if and only if their sequence of one bit measurements 
are the same.  See Figure~\ref{f:cell}.

%%%%%%%%%%%%%%% Figure
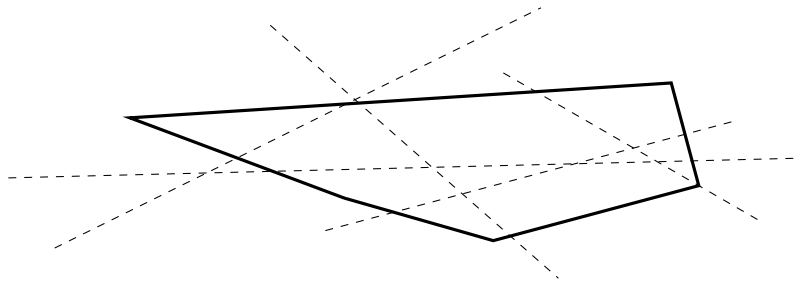
\begin{figure}
\begin{tikzpicture}[rotate=-30]
\draw[very thick] (0,-0.5) -- (2,0) --  (4,2) -- (3,3)  -- (-3,-1)  -- (0,-0.5);  
\draw[dashed] (1,2) -- (5,2);  
\draw[dashed] (0, -1) -- (4,3);   
\draw[dashed] (-3,-3) -- (1,3);  
\draw[dashed] (-2,1) -- (3,0);
\draw[dashed] (-4,-2.5) -- (5,3);  
\end{tikzpicture}

\caption{A region $ K$, bounded by the solid line, with an induced tessellation from hyperplanes represented by the 
dashed lines.}
\label{f:cell}
\end{figure}
%%%%%%%%%%%%%%% Figure

%%%%%%%%%%%%%%%%%%%%%%%%%%%%%%  DEFINITION DEFINITION DEFINITION
\begin{definition}\label{d:tesselation}  
We say that the tessellation of $ K$ induced by $ \{\theta _j\}$ is 
has \emph{$ \delta $-small cells} if for all $ x, y \in K$, if $  \textup{sgn} \langle x,\theta _j \rangle=\textup{sgn} \langle y,\theta _j \rangle$, for all $ j$, then $ d (x,y) < \delta $. Equivalently, all the cells of the induced tessellation have maximal 
diameter at most $ \delta $. 
\end{definition}
%%%%%%%%%%%%%%%%%%%%%%%%%%%%%%  DEFINITION DEFINITION DEFINITION

One can then ask how small the cells are in the induced tessellation of $ K$.    
The relevant characteristic  of $ K$ that will control the  this property is the 
\emph{metric entropy} $ N (K, \delta  )$,  i.e. the least number of $ d $-balls of radius $ \delta >0 $ 
 needed to cover $ K$.    
 
 We shall repeatedly make use of the following observation. Let $M(K, \delta)$ be the maximal cardinality  of a $\delta$-separated subset $K_0 \subset K$  (i.e. any  $x$, $y \in K_0$ with $x\neq y$ satisfy $d(x,y) > \delta$). Then the following obvious inequalities hold.
\begin{equation}\label{e:capacity}
M (K, 2\delta) \le N(K, \delta) \le M(K,\delta).
\end{equation}

 The next result  should  be compared to  \cite{MR3069959}*{Thm 4.2}, 
where the bound on  $ m$  is substantially larger, both in terms of  the power of $ \delta $, as well as in the fact that it involves the gaussian mean 
width, see \eqref{e:gmw},  instead of the smaller  metric entropy of $ K$.

%%%%%%%%%%%%%%%%%%%%%%%%%%%%%% THEOREM THEOREM THEOREM
\begin{theorem}\label{t:tessellation}[$ \delta $-Small Cells]  There are constants $ 0 < c < 1 < C$ so that for all integers $ n\ge 2$,  
for any $ K \subset \mathbb S^{n}$, and any $ 0 < \delta < 1$,  if $ m > C \delta ^{-1} \log N (K, c \delta )$, 
then with probability at least   $ 1 -  [ 2N (K, c\delta )]  ^{-2} $ the following  holds: 
%%  ENUMERATE
\begin{enumerate}
\item     The random vectors 
$  \{ \theta _j\} _{j=1} ^{m}$  induce  a   tessellation  of $ K$ with $ \delta $-small cells. 

\item  For any pair of points $ x, y\in K$, with $ \lvert  x-y\rvert > \delta  $, 
there are at least $ c \delta m$  choices of $ j$ such that 
\begin{equation*}
\langle x,  \theta _j \rangle < - c \delta < c \delta < \langle y, \theta_j   \rangle.  
\end{equation*}

\item  And,  for iid gaussians $ g_j$,  are  
 at least $ c \delta m$  choice of $ j$ with 
\begin{equation*}
\langle x,  g _{j } \rangle < - c \delta \sqrt n  < c \delta  \sqrt n  < \langle y, g _{j } \rangle.  
\end{equation*}

\end{enumerate}
%% ENUMERATE

\end{theorem}
%%%%%%%%%%%%%%%%%%%%%%%%%%%%%% THEOREM THEOREM THEOREM

Our proof is a modification of the standard `occupation time problem', which states  that it takes 
about $ n \log n$ independent uniform (0,1) observations to occupy all of the intervals 
$ [j/n, (j+1)/n)$, for $ 0 \leq j < n$.    The heuristic here is the following: \emph{the  small cells property 
 is governed by  metric entropy at scale $ \delta $.}

This heuristic  is quite useful in  the next  corollary, which improves the result of Plan and Vershynin \cite{MR3069959}*{Thm 2.1}, which proves the statement below with $ \delta ^{-3}$ replaced by $ \delta ^{-5}$.  
This  fact was basic to the main results of that paper concerning logistic regression in a compressed  one-bit setting.

%%%%%%%%%%%%%%%%%%%%%%%%%%%%%% COROLLARY COROLLARY COROLLARY
\begin{corollary}\label{c:} Let $ 0< s < n+1$ and $ 0< \delta < 1$ be given. Define a convex variant of the sparse vectors to be 
\begin{equation}\label{e:Kns}
K _{n,s} := \{  x \in \mathbb R ^{n+1} \;:\;  \lVert x\rVert_2 =1 , \lVert x\rVert_1 \leq s\}. 
\end{equation}
If $ m \gtrsim  \delta ^{ -3}s \log_+ \frac ns  $, then 
  the random vectors  $  \{ \theta _j\} _{j=1} ^{m}$  induce  a   tessellation  of $ K_{n,s}$ with $ \delta $-small cells.  
\end{corollary}
%%%%%%%%%%%%%%%%%%%%%%%%%%%%%%  COROLLARY COROLLARY COROLLARY

We now turn to questions that are \emph{multi-scale in nature}, namely  the \emph{restricted isometry property}, 
which will be abbreviated to RIP \cite{MR2300700}.

%%%%%%%%%%%%%%%%%%%%%%%%%%%%%%  DEFINITION DEFINITION DEFINITION
\begin{definition}\label{d:iso} Let  $ (X,d_X) $ and $ (Y, d_Y)$ be metric spaces.  
We say that map $ \varphi \;:\; X \mapsto Y$  has the 
$\delta $-RIP if 
\begin{equation*}
\lvert  d _{X} (x,x') - d _{Y} (\varphi (x), \varphi (x'))\rvert < \delta , \qquad x,x' \in X.  
\end{equation*}
\end{definition}
%%%%%%%%%%%%%%%%%%%%%%%%%%%%%%  DEFINITION DEFINITION DEFINITION

In many situations, the dimension of $Y$ is much smaller than that of $X$, i.e. restricted isometries are a form of dimension reduction.  In this direction, considering \emph{linear} isometries, 
Klartag and Mendelson  \cite{MR2149924} have proved 

%%%%%%%%%%%%%%%%%%%%%%%%%%%%%% THEOREM THEOREM THEOREM
\begin{priorResults}\label{t:KM}[$ \delta $-RIP into $ \ell ^{1}$]  There is a constant $ C>0$ so that for all integers $ n$,  $ K\subset \mathbb S ^{n}$, 
and $ 0 < \delta <1$,  if $ m \geq  C \delta ^{-2} \omega (K)$, with probability at least $ 1 - \operatorname {exp}(- C \delta ^{-2} \omega  (K) ^2  )$, the linear map $ \frac 1m A \;:\; \ell ^{2}_n \mapsto \ell ^{1}_m$, where the rows 
of $ A$ are iid gaussian vectors,   has the $ \delta $-RIP.
\end{priorResults}
%%%%%%%%%%%%%%%%%%%%%%%%%%%%%% THEOREM THEOREM THEOREM

Above, $ \omega (K)$ is the gaussian mean width, defined by 
\begin{equation}\label{e:gmw}
\omega (K) = \mathbb E \sup _{x, y \in K} \langle x - y, \gamma  \rangle, 
\end{equation}
where $ \gamma $ is a standard normal on $ \mathbb R ^{n}$. 
The motivation here comes from standard examples. Observe that the gaussian mean width of the sphere is 
\begin{equation*}
\omega (\mathbb S ^{n}) = \mathbb E \lvert  \gamma \rvert  \simeq  \sqrt n. 
\end{equation*}
Moreover, for the $ s$-sparse vectors $ K_s$ as in \eqref{e:sparse},  one has $ \omega (K_s) \simeq \sqrt {s \log \frac ns}$. 
That is, in   many cases  $ \omega (K) ^2 $ represents an intrinsic notion of dimension.  

To compare the bounds in the two different theorems, first observe that the metric for the gaussian process 
$ \gamma _x := \langle x, \gamma  \rangle$ satisfies 
\begin{equation} \label{e:g_metric}
\lVert \gamma _x - \gamma _y \rVert_2 
= \bigl[ \mathbb E \langle x-y, \gamma  \rangle ^2  \bigr] ^{1/2} \simeq d (x,y).  
\end{equation}
Second,  use  Theorem~\ref{t:sudakov}, Sudakov's lower bound for the supremum of 
gaussian processes,    to see that  
\begin{equation*}
\delta ^{-1} \log N (K,  \delta ) = \delta ^{-3} [ \delta^2 \log N (K, \delta)] 
\lesssim \delta ^{-3} \omega (K) ^2 .  
\end{equation*}
See Schechtman \cite{MR2199631} 
for a detailed discussion of the theorem above, its relation to Dvoretsky Theorem, and remarks about optimal bounds.  

\medskip 

We turn to the subject of  the  one bit RIP.  
Namely, the construction of  $ \delta $-RIPs  from $ K \subset \mathbb S ^{n}$ into a  Hamming cube of dimension $ m$.  
Plan and Vershynin \cite{MR2149924}  proved that if $  m \gtrsim \delta ^{-6} \omega (K) ^2 $, 
than the one bit map \eqref{e:sgnA} satisfies  $ \delta $-RIP from the sphere into the Hamming cube, with 
high probability.  
They also  suggested that $ \omega (K) ^2 $ would also be the correct measure of the size of $ K$ for 
$ \delta $-isometries into the Hamming cube. 
We cannot verify this, but can prove an estimate of the form $ m \gtrsim \delta ^{-2} H (K) ^2 $, where 
$ H (K)$ is a different measure of the intrinsic dimension of $ K$.   In the case of sparse 
vectors,  we will see that $ H (K_s) \approx \omega (K_s) \approx \sqrt {s \log \frac ns}$.   

To set the stage for this new measure, define  $ H_x := \{\theta \in \mathbb S^{n} \;:\; \langle \theta ,x \rangle \ge 0\}$ 
to be  the positive hemisphere relative to $ x$.    
Observe that a hyperplane $\theta^\perp$  point  $ \theta  $ separates two points $ x, y \in \mathbb S^{n}$ 
if and only if $ \theta \in W _{x,y} := H_x \triangle H_y$ ($W$ stands for `wedge'). 
The symmetric difference between 
two hemispheres,  is related to the geodesic distance on the sphere through the essential equality 
$ \mathbb P (W _{x,y}) = d (x,y)$, which is a simple  instance of  the Crofton formula.   See Figure~\ref{f:W} and  \cite{MR0333995}*{p. 36-40}. 

Let us use randomly selected  points  $ \{ \theta _1 ,\dotsc, \theta _m\} \subset \mathbb S ^{n}$ 
to define the one bit map given by \eqref{e:sgnA}, namely 
\begin{equation*}
\varphi (x) = \{  \textup{sgn} (\langle x , \theta _j \rangle) \;:\; 1\le j \leq m\} \in \{-1, +1\} ^{m} . 
\end{equation*}
Restricting this map to  $ K\subset \mathbb S ^{n}$, the $ \delta $-RIP property becomes 
\begin{equation*}
\sup _{x,y\in K}   \Bigl\lvert\frac 1m  \sum_{j=1} ^{m}  \mathbf 1_{W _{x,y}}
(\theta _j) - \mathbb P (W _{x,y})\Bigr\rvert\leq \delta . 
\end{equation*}
That is, we need to bound, in non-asymptotic fashion, a standard empirical process over the class of wedges $ W _{x,y}$ defined by $ K$. 
Essential to such an endeavor is to understand the asymptotic  behavior of the associated empirical process.

%%%%%%%%%%%%%%% Figure
\begin{figure}
\tdplotsetmaincoords{65}{110}

\pgfmathsetmacro{\rvec}{1.2}

\pgfmathsetmacro{\thetavecc}{55}
\pgfmathsetmacro{\phivecc}{35}

\pgfmathsetmacro{\thetaveccc}{39.7}
\pgfmathsetmacro{\phiveccc}{55}

\begin{tikzpicture}[scale=2,tdplot_main_coords]

\shadedraw[tdplot_screen_coords,ball color = white] (0,0) circle (\rvec);

%-----------------------
\coordinate (O) at (0,0,0);

\tdplotsetcoord{B}{\rvec}{\thetavecc}{\phivecc}

\tdplotsetcoord{C}{\rvec}{\thetaveccc}{\phiveccc}

%draw the main coordinate system axes
\draw[thick,->] (0,0,0) -- (1.05,-.6,0) node[anchor=north east]{$x$};
\draw[thick,->] (0,0,0) -- (1.2, -0.2,0) node[anchor=north east]{$y$};
%\draw[thick,->] (0,0,0) -- (0,0,1.7) node[anchor=south]{$z$};

\draw (0,.5,0) node {$ W _{x,y}$}; 

%\draw[-stealth,very thick,color=blue] (O) -- (B);
%

%\draw[-stealth,very thick,color=green!60!black] (O) -- (C);

%\draw[dashed, color=blue] (O) -- (Bxy);
%\draw[dashed, color=blue] (B) -- (Bxy);
%\draw[dashed, color=green!60!black] (O) -- (Cxy);
%\draw[dashed, color=green!60!black] (C) -- (Cxy);

%\tdplotdrawarc[color=blue]{(O)}{0.3}{0}{\phivecc}{anchor=north}{$\lambda_A$}

\tdplotsetthetaplanecoords{\phivecc}

%\tdplotdrawarc[color=blue,tdplot_rotated_coords]{(0,0,0)}{0.3}{90}{\thetavecc}{anchor=south west}{$\varphi_A$}

%\tdplotdrawarc[color=green!40!black]{(O)}{0.7}{0}{\phiveccc}{anchor=north}{$\lambda_B$}

\tdplotsetthetaplanecoords{\phiveccc}

%\tdplotdrawarc[color=green!40!black,tdplot_rotated_coords]{(0,0,0)}{0.7}{90}{\thetaveccc}{anchor=south west}{$\varphi_B$}
\draw[dashed] (\rvec,0,0) arc (0:360:\rvec);
\draw[thick] (\rvec,0,0) arc (0:110:\rvec);
\draw[thick] (\rvec,0,0) arc (0:-70:\rvec);

\tdplotsetthetaplanecoords{35}
\draw[thick,tdplot_rotated_coords] (\rvec,0,0) arc (0:151:\rvec);
%\draw[very thick,color=red,tdplot_rotated_coords] (\rvec,0,0) arc (0:55:\rvec);
\draw[dashed,tdplot_rotated_coords] (\rvec,0,0) arc (180:-40:-\rvec);
\draw[thick,tdplot_rotated_coords] (\rvec,0,0) arc (360:336:\rvec);

\tdplotsetthetaplanecoords{55}
\draw[thick,tdplot_rotated_coords] (\rvec,0,0) arc (0:147:\rvec);
%\draw[very thick,color=red,tdplot_rotated_coords] (\rvec,0,0) arc (0:40:\rvec);
\draw[dashed,tdplot_rotated_coords] (\rvec,0,0) arc (180:-40:-\rvec);
\draw[thick,tdplot_rotated_coords] (\rvec,0,0) arc (360:334:\rvec);

\tdplotsetrotatedcoords{-79.1}{-120}{27.3}
%\draw[very thick,color=red,tdplot_rotated_coords] (\rvec,0,0) arc (0:21:\rvec);

\end{tikzpicture}

\caption{An illustration of the wedge $ W _{x,y} = H_x \triangle H_y$. }
\label{f:W}
\end{figure}
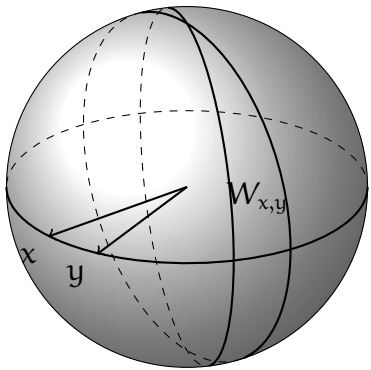
%%%%%%%%%%%%%%% Figure

We state the following theorem without proof.  It is an asymptotic result, while our emphasis is on 
non-asymptotic ones. However, it points to a distinguished role played by the gaussian process that appears 
in the conclusion.  We rely upon the well-known fact 
that a gaussian process is uniquely determined by its mean and covariance structure.  

%%%%%%%%%%%%%%%%%%%%%%%%%%%%%% THEOREM THEOREM THEOREM
\begin{theorem}\label{t:weak}  As $ m \to \infty $, the empirical process indexed by hemispheres converges weakly 
\begin{equation*}
\frac 1 {\sqrt m} \sum_{j=1} ^{m}  \mathbf 1_{H_x} (\theta _j) - \tfrac 12  \stackrel{d} \longrightarrow G_x 
\end{equation*}
where $ x$ ranges over $ \mathbb S ^{n}$, and $ G_x$ is the mean-zero gaussian process with 
\begin{equation}\label{e:hemi}
\mathbb E G_x ^2 = \tfrac 1 {4} , \qquad  (\mathbb E  (G_x -G_y) ^2) ^{1/2}  =  \mathbb P ( W_{x,y}) ^{1/2} = \sqrt{d (x,y)}. 
\end{equation}
\end{theorem}
%%%%%%%%%%%%%%%%%%%%%%%%%%%%%% THEOREM THEOREM THEOREM

We call the gaussian process above the   \emph{hemisphere gaussian process}. It is  a key innovation of this paper. 
Particularly relevant in comparing our results with those stated in terms of 
gaussian mean width, note that the metric  for $ \gamma _x$ \eqref{e:g_metric} and that of $ G_x$ are related (essentially) through a square root.  

We then define the  \emph{hemisphere mean width of $ K\subset \mathbb S ^{n}$ } by 
\begin{equation} \label{e:Hdef}
H(K) := \mathbb E \sup _{x,y\in K}   \big( G_x - G_y \big)
\end{equation}
It always holds that $ \omega (K) \lesssim H (K) \lesssim \sqrt {n}$, and for the $ s$-sparse vectors $ K_s$, 
$ \omega (K) \simeq H (K) = \sqrt {s \log_+ \frac ns}$. 
We comment more on this measure   in \S \ref{s:H} below.

The next result concerns the bounds for the  one-bit $ \delta $-RIP.  It has two parts, one for general $ K\subset \mathbb S ^{n}$, 
and the other for the unit sparse vectors   defined  in \eqref{e:sparse}.

%%%%%%%%%%%%%%%%%%%%%%%%%%%%%% THEOREM THEOREM THEOREM
\begin{theorem}\label{t:H}[$ \delta $-RIP into Hamming Cube] 
There is a constant $ C>0$ so that for all dimensions $ n\in \mathbb N$,  subsets $ K\subset \mathbb S ^{n}$, 
and $ 0 < \delta <1$,  the following  two conclusions hold. 
%%  ENUMERATE
\begin{description}
\item [General $ K$] There exists a $ \delta $-RIP map from  $ (K, d) $  to the 
Hamming cube  $ ( \{-1, 1\} ^{m}, d _{H})$, provided   that $   m\geq C  \delta ^{-2}  H (K) ^2 $. 

\item [Sparse vectors] For $ 1\leq s \leq n+1$, with probability $ 1 - \operatorname {exp}\bigl(- s \log_+ \frac ns\bigr)$, the map $ x \mapsto  \textup{sgn} (A x)$ has $ \delta $-RIP from  $ (K_s, d) $  to the 
Hamming cube  $ ( \{-1, 1\} ^{m}, d _{H})$, provided   that $   m\geq C  \delta ^{-2} s \log_+ \frac ns   $.  
\end{description}
%% ENUMERATE
\end{theorem}
%%%%%%%%%%%%%%%%%%%%%%%%%%%%%% THEOREM THEOREM THEOREM

Notice that for general $ K$, we do not assert that $ x \mapsto  \textup{sgn} (A x)$ is the  isometry, 
whereas Plan and Vershynin    \cite{MR3164174}*{Thm 1.2} show that it is, but with  $ m \gtrsim \delta ^{-6} \omega (K) ^2 $. 
We comment more on this in \S 3.    
But, note that in the sparse vector case, we only require  $ m$ to be as big as in the 
linear $ \delta $-RIP. That is we have this surprising conclusion. 
\begin{quote}
\emph{For sparse vectors   the one-bit map is just as effective as the linear map.}
\end{quote}
We note that \cite{13051786} very nearly proves this result for sparse vectors, missing by a factor of $ \log 1/ \delta $, 
and using a proof that is more involved than ours.  
Also, the Jacques \cite{150406170} considers quantized maps from sets $ K\subset \mathbb R ^{n}$ into $ \mathbb Z ^{m}$, 
using gaussian mean width. It seems plausible that the hemisphere process is relevant to that paper as well.

This corollary, in which $ K$ is a finite set,  is immediate, and is a one bit analog of the Johnson-Lindenstrauss lemma.

%%%%%%%%%%%%%%%%%%%%%%%%%%%%%% COROLLARY COROLLARY COROLLARY
\begin{corollary}\label{c:embeds} 
If $ K\subset \mathbb S ^{n}$ is finite, then    there is a $\delta $-isometry 
from  $ (K, d) $ into  the Hamming cube $  ( \{-1, 1\} ^{m}, d _{H})$, where $ m \lesssim \delta ^{-2} \log \lvert  K\rvert  $.   
\end{corollary}
%%%%%%%%%%%%%%%%%%%%%%%%%%%%%%  COROLLARY COROLLARY COROLLARY

Recall that the Johnson-Lindenstrauss lemma states that a  finite set  $ K\subset \mathbb R ^{n}$ 
has a linear Lipschitz embedding into $ \mathbb R ^{m}$, where $ m \gtrsim  \delta ^{-2} \log \lvert  K\rvert $. 
Namely, for linear $ A \;:\; \mathbb R ^{n} \mapsto \mathbb R ^{m}$, 
\begin{equation*}
\bigl\lvert \lVert Ax - A x'\rVert_2 - \lVert x-x'\rVert_2\bigr\rvert\leq \delta \lVert x-x'\rVert_2. 
\end{equation*}
Here, $ A$ is   $ \frac 1m$ times a standard $ m \times n$ gaussian matrix, the inequality holding with high probability.    
Embeddings of $ \ell ^2 $ into $ \ell ^{1}$ are of significant interest (\cite{MR2755868} and references therein), from the perspective of the best possible bounds, as well as implementations with few random inputs.

Another relevant property is \emph{sign-product RIP} property.  
This we specialize to the case of sparse vectors, and also add modest restrictions on $ s$ and $ \delta $ in order 
get a value of $ m$ that matches those of the other Theorems.   Namely, we require a weak lower bound on $ \delta $, 
as a function of $ s/n$.  

%%%%%%%%%%%%%%%%%%%%%%%%%%%%%% THEOREM THEOREM THEOREM
\begin{theorem}\label{t:sign_product}  
   There is a constant C so that 
for all integers $ 0 < s \leq n $,  and $0 < \delta < 1 $ such that 
\begin{equation} \label{e:assume}
 ( s/ 10n) ^{4}< \delta   
\end{equation}
with probability at least $ 1 - (\frac s {2n}) ^{2s}$, 
for all $  m   \simeq   C \delta ^{-2} s \log_+ \frac ns $, there holds 
\begin{equation} \label{e:sign_product}
\sup _{x, y \in K_s}  
\frac 1m \Bigl\lvert 
\sum_{j=1} ^{m}   \textup{sgn} (\langle x, g_j \rangle) \langle y,g_j \rangle - \lambda \langle x,y \rangle 
\Bigr\rvert \leq \delta 
\end{equation}
where the $ g_j$ are iid standard gaussians on $ \mathbb R ^{n+1}$, and $ \lambda = \sqrt {\frac 2 \pi }$.  
\end{theorem}
%%%%%%%%%%%%%%%%%%%%%%%%%%%%%% THEOREM THEOREM THEOREM

This RIP property arises from \cite{13051786}, and a quantitative estimate $ m \gtrsim \delta ^{-6} \omega (K) ^2 $ 
is proved in  \cite{MR3008160}*{Prop. 4.3}  for general $ K$.  Furthermore, this latter bound is heavily used in 
\cite{14078246}*{\S4} in the sparse vector case.  
Above we have seemingly optimal lower bounds on $ m$, subject to mild restrictions on  $ \delta $.

\bigskip

One can ask what is the best value of $ m = m (K, \delta )$ that can achieve the bounds in these different 
estimates, from small cells to the restricted isometry properties.   The estimates of Plan and Vershynin \cite{MR3164174} were in terms of the 
gaussian mean width. 
  Using Sudakov's lower bound, for the supremum of 
gaussian processes,  Theorem \ref{t:sudakov}, note that 
\begin{equation} \label{e:compare}
\sqrt {\log N (K, \delta )} \lesssim 
\begin{cases}
 \delta ^{-1} \omega (K) 
 \\
 \delta ^{-1/2} H (K) 
\end{cases}
\end{equation}
with the difference in the right hand side being a consequence in the difference of a square root in the metrics 
of the two gaussian processes.  Indeed, $N(K, d_G, \delta )  = N (K,\delta^2)$,  where $d_G$ is the metric of the hemisphere process. Therefore, 
\begin{equation*}
\delta ^{-1} \log N (K,  \delta )  \lesssim 
\begin{cases}
\delta ^{-3} \omega (K) ^2 
\\
 \delta ^{-2}H (K) ^2 
\end{cases}
\end{equation*}
The left hand side is the bound for the $ \delta $-small cell property. 
The lower part of the right hand side is the  bound required for the $ \delta $-RIP into the Hamming cube.  
It would appear that $ \delta ^{-3}$ is the smallest multiple of $ \omega (K) ^{2}$ that could appear in such theorems, 
and, for instance,  the  bounds  of
Plan and Vershynin \cite{MR3164174} were of the form  $ \delta ^{-6} \omega (K ) ^2 $.

\bigskip 

The  small cells Theorem~\ref{t:tessellation}, proved in \S~\ref{s:tess},  will be seen as a consequence of a standard occupation time calculation. 
The RIP Theorem~\ref{t:H} is a calculation involving  empirical processes, after reducing a general $ K $ to a finite 
approximating set.   Specializing $ K$ to sparse vectors leads naturally to the notion of VC dimension. 
This is detailed in \S~\ref{s:H}, with the sign-product RIP Theorem~\ref{t:sign_product} proved in \S~\ref{s:sign}. 
Powerful inequalities for empirical processes for VC sets simplify the analysis of the these theorems. The background information on empirical processes
and the properties we need are collected  in \S~\ref{s:prob}.

\begin{ack}
This work was completed as part of a the program in High Dimensional Approximation,  Fall 2014, at 
ICERM, at Brown University.  We thank Simon Foucart for asking us about the sign-product RIP.  
\end{ack}

%%%%%%%%%%%%%%%%%%%%%%%%%%%%%% SECTION  SECTION SECTION
%%%%%%%%%%%%%%%%%%%%%%%%%%%%%% SECTION  SECTION SECTION 
\section{Small Cells: Proof of Theorem~\ref{t:tessellation}} \label{s:tess}
We introduce  the notion of a hyperplane \emph{transversely separating} vectors $ x,y \in \mathbb S ^{n}$. 
Assuming $ x$ and $ y$ are not antipodal, there is a unique geodesic  $ \tau $  from $ x$ to $ y$. 
Any hyperplane  $ H$ that separates $ x$ and $ y$ must intersect $ \tau $.  We further say that \emph{$ H$ transversely separates $ x$ and $ y$} if  (a) the angle  of intersection 
between $ \tau $ and $ H$  is at least $ \pi /4$, and (b) that the distance of the point of intersection $ \tau \cap H$ 
to both $ x$ and $ y$ be at  least $ d (x,y) /4$.  The heuristic motivation for this definition is that if a hyperplane transversely separates $x$ and $y$, it also has to separate point lying close to $x$ and $y$, which allows one to pass to a finite subset.  Let 
\begin{equation}\label{e:transversal}
\widetilde W _{x,y} := \{  \theta \in W _{x,y} \;:\;   \textup{$ \theta ^{\perp}$  transversely separates $ x$ and $ y$} \}. 
\end{equation}
Now, for a randomly selected $ \theta  \in W _{x,y}$, 
the angle of intersection between $ \theta ^{\perp}$ and  $ \tau $ is uniformly distributed on the circle, as is the point of intersection $ \theta ^{\perp} \cap \tau $ on $ \tau $.  
Moreover, these two quantities are statistically independent.  From this, it follows that 
\begin{equation*}
\mathbb P (\widetilde W _{x,y})  =  \tfrac 14  d (x,y). 
\end{equation*}

%%%%%%%%%%%%%%% Figure
\begin{figure}
\begin{tikzpicture}
\filldraw (0,0) circle (.05cm) node[below] {$ x$};
\filldraw (3,0) circle (.05cm) node[below] {$ y$};
\draw[thick] (-1,.7) -- (3.5, -1) node[right] {$  \theta ^{\perp}$}; 
\draw[dashed] (-1.25, -.3) -- (3.5,.5) node[right] {$\tilde  \theta ^{\perp}$};  
%\draw (0,0) node[below] {$ x$} sin (1.57,1) node[below] {$ y$}; 
%\draw (0 , 1 ) -- (2, 0 ) node[right] {$ H$}; 
\end{tikzpicture}

\caption{The points $ x$ and $ y$ are transversely separated by $  \theta ^{\perp}$, but not $ \tilde  \theta ^{\perp}$.}
\label{f:trans}
\end{figure}
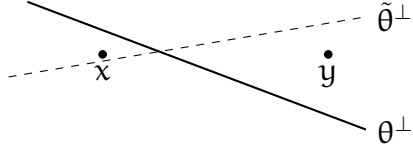
%%%%%%%%%%%%%%% Figure

The main line of argument begins with  selection of 
 $ K _{0} \subset K$,   a maximal subset such that each pair of distinct points $ x, y \in K_0$ 
satisfies $ d (x,y) \ge \delta /64$.   Then, as is well known, see \eqref{e:capacity}, $ \lvert   K_0\rvert  \le N (K, \delta /128)$. 

\smallskip 
Observe the following: if  $ \mathcal H$  is a collection of hyperplanes for which
each pair of distinct  elements of $ K _{0}$ is transversely separated  by a  hyperplane in $ \mathcal H$, 
then   $ \mathcal H$  induces a tessellation with  $ \delta $-small cells on $ K$. 
Indeed,  for two points $ x,y \in K$, separated by $ \delta $, let $ \tau $ be the geodesic 
between $ x$ and $ y$, and then, take  $ x' , y'\in K _0$  to be the points closest to $ x$,  and $ y$, respectively. 
So, $ d (x,x'), d (y,y')\leq \tfrac 1 {64} \delta $, and $ d (x',y') \ge \frac {31} {32} \delta $.   Let $ \tau '$ be the 
geodesic between $ x'$ and $ y'$.  

Now, $ x'$ and $ y'$ are transversely separated by  a hyperplane $ H \in \mathcal H$, by assumption.  
Then,  (i) the point $ z' = H \cap \tau '$ is at distance at most $ \delta /8 $  from $ \tau $, 
(ii) $ z'$  is at least  at distance $ \frac 1 5 d (x,y)$ from both $ x$ and $ y$, 
and hence (iii) $ H$ separates $ x$ and $ y$, and moreover, the point $ z = \tau  \cap H $ satisfies $ d (z,x) \ge \frac 1 {20} d (x,y)$, 
and similarly for $ y$.  

Indeed, to see (i) the lengths of $ \tau $ and $ \tau '$  very close. 
\begin{equation*}
\lvert  d (x,y) - d (x',y')\rvert \le \tfrac 1{32} \delta .%= \tfrac 1 {32} \frac {\delta } {d (x,y)} \cdot d (x,y).
\end{equation*}
Let $ \tau ''$ be the geodesic obtained by a rigid motion of $ \tau '$ so that $ \tau ''$ has $ x$ as an endpoint. 
Then, parameterizing $ \tau $ and $ \tau ''$ by arc length $ s$,  and identifying $ \tau (0) = \tau '' (0)=x$, 
the quantity $ d (\tau (s), \tau '' (s))$ is monotone increasing.  
But, at $ s= d (x,y)$, we will have 
\begin{align*} 
 \textup{dist} (z', \tau ) &\leq   \tfrac 1 {64} \delta +   d (\tau (s), \tau '' (s)) 
\\ &=  \tfrac 1 {64} \delta+ d (y, \tau '' (s))  
\leq  \tfrac 1 {64} \delta + d (y, y') + d (y', \tau '' (s)) 
\\
&\leq \tfrac 1 {32} \delta  + d (y', \tau '' (d (x',y')))  + \lvert  d (x',y') - d (x,y)\rvert \leq \tfrac 1 { 8} \delta .   
\end{align*}
To see (ii), just use the triangle inequality to get a better result.  
\begin{equation*}
d (z', x) \ge d (z', x') - d (x',x) \ge \tfrac 1 4  d (x',y') - \tfrac 1 {64} \delta \ge \tfrac {29} {128} d (x,y). 
\end{equation*}
And,   (iii) would be obvious on the plane. We are however on the sphere,  
but with spherical triangles of small diameter.   The spherical corrections are 
small, so the result will follow. 

\smallskip 

We can now argue for the first conclusion of the Theorem, concerning small cells. 
Take $ \mathcal H = \{  \theta  _  \ell ^{\perp} \;:\;  1\le \ell \leq m \}  $ to be  iid uniform random hyperplanes,  
where $ m \ge C \delta ^{-1} \log N (K, \delta /128)$. 
For  a pair of points $ x , y \in K_0$,   $x\neq y$, the chance that no $ \theta _ \ell $ transversely separates  
$ x$ and $ y$ is, for some fixed $ 0< c < 1$, 
\begin{align*}
\mathbb P ( \theta _{\ell } \not\in \widetilde W_ {x,y}, \  1\leq \ell \leq m) &\leq (1 -  c\delta ) ^{m} 
\\&\leq \operatorname {exp}(  -c'\delta m) \leq (2 \cdot  N (K , \delta /128)) ^{-4}, 
\end{align*}
with appropriate choice of $ C$. Now, there are at most $  N (K , \delta /128) ^2 $ pairs of points $  x , y \in K_0$, $x\neq y$,
hence, by the union bound, the probability that $ K_0$ is not  transversely separated by $ \mathcal H$ is 
at most  $ (2 \cdot  N (K , \delta /128)) ^{-2} $. So the proof  of the first conclusion of  Theorem \ref{t:tessellation} is complete.  

\medskip 

For the second conclusion of the Theorem,  begin by observing that if 
$ \theta ^{\perp}$ transversely separates $ x$ and $ y$, then, 
we have necessarily  
\begin{equation} 
\langle x, \theta  \rangle < - c_0 d (x,y) < c_0 d (x,y) < \langle y, \theta  \rangle, 
\end{equation}
where above, we can take $ c_0 = \frac 1 {10}$.  
Above, we assume as we may that $ \langle x, \theta  \rangle< 0$.   

We will then prove the conclusion of the Theorem for $ K_0$.  Namely if 
 if $ m > C \delta ^{-1} \log N (K, c \delta )$, 
with probability at least   $ 1 -  [ 2N (K, c\delta )]  ^{-2} $, 
for any pair of points $ x, y\in K$, with $ \lvert  x-y\rvert > \delta  $, 
there are at least $ c' \delta m$  choices of $ j$ such that 
\begin{equation} \label{e:0far}
\langle x,  \theta _j \rangle < - c_0 \delta < c_0 \delta < \langle y, \theta_j   \rangle.  
\end{equation}
We can then draw the same conclusion for all of $ K$, provided we replace $ c_0$ above by $ \frac 1 {20}$, 
by   property (ii) above.   

Now, $ p _{x,y}=\mathbb P (\theta \in \widetilde W _{x,y}) > \frac 1 {256}\delta $,   so for sufficiently small $ c$, 
\begin{align*}
\mathbb P \Bigl(\sum_{j=1} ^{m}   \mathbf 1_{\theta _j \in \tilde W _{x,y}} < c \delta m  \Bigr) 
&\leq \mathbb P \Bigl(\sum_{j=1} ^{m}   \mathbf 1_{\theta _j \in \tilde W _{x,y}}  - p _{x,y}    <  - \tfrac c2  \delta m  \Bigr)  
\\
& \lesssim \operatorname {exp}\Bigl( - \frac { c ^2 \delta ^2 m ^2 } { c' \delta  m }\Bigr) 
\lesssim \operatorname {exp}( -  C \log N (K, \delta /64))  
\end{align*}
by choice of $ m$.  We can make $ C$ large, so that a simple union bound shows that with probability at least 
$ 1 - (2 N (K, \delta /64)) ^{-4}$, for each pair of $ x,  y \in K_0$, $x\neq y$, there are at least $ c \delta m $ choices of $ j$ for 
which $ \theta _j$ satisfies \eqref{e:0far}.  The proof is complete. 
\medskip 

For the third conclusion of the Theorem, we want to make the same conclusion as in the the second part, 
but for gaussian r.v.s.  But, we know the result for $ \theta _j$ chosen at random on the sphere.  
And, when making a gaussian observation   $ g$  on $ \mathbb R ^{n+1}$, we say that it is \emph{moderate} if $  \tfrac   {\sqrt {n+1}}2 \le \lvert  g\rvert \leq 2 \sqrt {n+1}  $. 
Note that 
\begin{equation*}
\mathbb P ( \tfrac  {\sqrt {n+1}}2 \le \lvert  g\rvert \leq 2 \sqrt {n+1} ) = 
\mathbb P \Bigl( -\frac {3(n+1)} 4  \leq \sum_{\ell =1} ^{n+1} g_ \ell ^2 -1   \leq 3 (n+1)   \Bigr) 
\end{equation*}
where the $ g _{\ell }$ are one dimensional standard gaussian r.v.s.  It follows that 
the probability of $ g$ being moderate tends to one, as $ n $ tends to $ \infty $.  

Now, for iid gaussians $ \{g_j\}$ taking values in $ \mathbb R ^{n+1}$, set $ \theta _j  = g_j / \lvert  g_j\rvert $. 
If $ \theta _j $ satisfies \eqref{e:0far}, and $ g_j$ is moderate, then $ g_j/ \sqrt {n+1}$ satisfies \eqref{e:0far}, with different 
constants.  
The event that $ g_j$ being moderate, and the distribution of $ \theta _j$ are statistically independent, 
hence we can condition on the $ \{g_j\}$ being moderate, and apply the second conclusion, to deduce our result.

%%%%%%%%%%%%%%%%%%%%%%%%%%%%%% SUBSECTION SUBSECTION SUBSECTION SUBSECTION
 %%%%%%%%%%%%%%%%%%%%%%%%%%%%%% SUBSECTION SUBSECTION SUBSECTION SUBSECTION 
\subsection{Proof of Corollary~\ref{c:}}%\label{ss.}

The next lemma improves upon \cite{MR3069959}*{Lemma 3.4} and, together with Theorem~\ref{t:tessellation}, clearly proves 
the Corollary. 

%%%%%%%%%%%%%%%%%%%%%%%%%%%%%% LEMMA LEMMA LEMMA
\begin{lemma}\label{l:} The set $ K _{n,s}$ defined in \eqref{e:Kns} satisfies the metric entropy bound 
\begin{equation*}   
\delta ^2  \log N (K _{n,s}, \delta ) \lesssim  s   \log_+ n/s. 
\end{equation*}
\end{lemma}
%%%%%%%%%%%%%%%%%%%%%%%%%%%%%% LEMMA LEMMA LEMMA

%%%%%%%%%%%%%%%%%%%%%%%%%%%%%% PROOF PROOF PROOF
\begin{proof}
To estimate this quantity, we use Sudakov's bound in the form of the upper estimate in \eqref{e:compare}. 
Thus, we have 
\begin{equation*}
\delta ^2  \log N (K _{n,s} , \delta ) \lesssim  \omega  (K _{n,s}) ^2 . 
\end{equation*}
 
Therefore, it suffices to show  that $ \omega  (K_{n,s}) ^2 \lesssim s \log_+ \frac ns $, which is a well-known fact. 

This follows from: 
(A)  the convexity result  $ K _{n,s} \subset 2 \operatorname{convex}(K_s)$, where the latter is the 
convex hull of $ K_s$;  (B) the further convexity estimate 
$ \omega (\operatorname{convex} (K)) \lesssim \omega (K)$; and (C) $ \omega  (K_s) ^2 \lesssim s \log_+ \frac ns$.

These three points are as follows. The point (A) 
 is seen by a common technique in compressed sensing.  
Given $ x \in K _{n,s}$, write it as $ x= \sum_{j} x_j$, where $ x_1$ is $ x$ restricted to the $ s$-largest 
coefficients, $ x_2$ has the next $ s$-largest coefficients, and so on.  It suffices to show that 
\begin{equation*}
\sum_{j} \lVert x_j\rVert_2 \leq 2. 
\end{equation*}
Now, $ \lVert x_1\rVert_2 \leq 1$.  And, for $ j >1$, each coordinate of $ x _{j}$ is dominated by the 
least coordinate of $ x _{j}$, which is at most  $ \lVert x _{j-1}\rVert_1/s$.  
Hence $ \lVert x_j\rVert_2 \leq \lVert x _{j-1}\rVert_1/\sqrt s $. But, we have a bound on  $\lVert x\rVert_1 $, hence 
\begin{align*}
\sum_{j=1} ^{\infty }  \lVert x_j\rVert_2 &\leq \lVert x_1\rVert_2 +\sum_{j=2} ^{\infty }  \lVert x_j\rVert_2
\\
& \leq 1 + \frac 1 {\sqrt s}  \sum_{j=1} ^{\infty }  \lVert x_j\rVert_1 \leq 2. 
\end{align*}

Point (B) is a deep implication of Talagrand's majorizing measure theorem.  And (C) is a well-known fact, 
see for instance \cite{MR3008160}*{Lemma 2.3}. 

\end{proof}
%%%%%%%%%%%%%%%%%%%%%%%%%%%%%% PROOF PROOF PROOF

%%%%%%%%%%%%%%%%%%%%%%%%%%%%%% REMARK REMARK REMARK
\begin{remark}\label{r:delta}   
We have chosen the proof above for convenience.  
The gaussian mean width is used, namely the   upper half of the Sudakov estimate \eqref{e:compare}, so that we 
could appeal to Talagrand's convexity inequality. (With the hemisphere process, 
the convexity argument would be more complicated.)
Thus, the power on $ \delta $ we get is $ \delta ^{-3}$, which as \eqref{e:compare} suggests, is optimal for 
this strategy. 
Potentially, there is an additional 
improvement in the power of $ \delta $, but we do not pursue it here. 

\end{remark}
%%%%%%%%%%%%%%%%%%%%%%%%%%%%%% REMARK REMARK REMARK

%%%%%%%%%%%%%%%%%%%%%%%%%%%%%% SECTION  SECTION SECTION
%%%%%%%%%%%%%%%%%%%%%%%%%%%%%% SECTION  SECTION SECTION 
\section{$ \delta $-RIP into the  Hamming Cube: Proof of Theorem~\ref{t:H}} \label{s:H}

%%%%%%%%%%%%%%%%%%%%%%%%%%%%%% SUBSECTION SUBSECTION SUBSECTION SUBSECTION
 %%%%%%%%%%%%%%%%%%%%%%%%%%%%%% SUBSECTION SUBSECTION SUBSECTION SUBSECTION 
\subsection{The Case of General $ K\subset \mathbb S ^{n}$}%\label{ss.}
The distinction between the case of general $ K$ and sparse $ K_s$ is that 
in the general case, we avoid making an entirely specific choice of RIP.  
Let $ K_0$ be a maximal cardinality subset of $ K$ so that  
for all $ x\neq y \in K_0$,  one has $ d (x,y) >  \delta/4$. 
Of course $ \lvert  K_0\rvert \leq  N (K, \delta/8 ) $, see \eqref{e:capacity}.  
Letting $ \pi _0 \;:\; K \mapsto K_0$ be the map sending $ x\in K$ to the element of $ K_0$ closest to $ x$, 
one then has $ d (x, \pi _0 x) \le \delta/4  $.  That is $ \pi _0 \;:\; K \mapsto K_0$ has  $ \delta /2$-RIP.

We construct a  one-bit map  $ \varphi _0 \;:\; K_ {0} \mapsto \{ -1,1\} ^{m}$, which is also a $ \delta /2$-RIP map, i.e. 
\begin{equation}\label{e:Ks}
\sup _{x,y\in K _{0}}  \bigl\lvert   d_H( \varphi_0 (x) , \varphi_0 (y))  - d (x,y)\bigr\rvert \leq \delta/2. 
\end{equation}
Observe that if we extend $ \varphi _0$ to $ K$ by the formula $ \varphi (x) := \varphi _0 (\pi _{s_0} x)$, 
we will have then proved  Theorem \ref{t:H}, since the composition of a $ \delta _1$-RIP  and a $ \delta _2$-RIP  is a $ \delta _1 + \delta _2$-RIP. 

The mapping $ \varphi _0$ is natural map  $ x \mapsto \textup{sgn} (A x)$, as in \eqref{e:sgnA}, but restricted to $ K_0$.  
Note that  for $ x,y\in K_0$ the Hamming distance is 
\begin{align}
d _{H} ( x,y) 
& = \frac 1 m 
\sum_{\ell =1} ^{m} \mathbf 1 \{  \textup{ $ \theta  _{\ell } ^{\perp} $ separates $ x$ and $ y$ }\} 
\\  \label{e:2Empirical}
& = \frac 1 m \sum_{\ell =1} ^{m} \mathbf 1 _{W _{x,y}} (\theta _{\ell }).  
\end{align}
In expectation, this is the geodesic distance:  $\mathbb E d_H (x,y) =  d (x,y)$.   We will show that 
\begin{equation}  \label{e:dHK}
\mathbb E \sup _{ x, y \in K_0}  \lvert d _{H} ( x,y)  - d (x,y) \rvert \leq   H (K)/\sqrt m 
\end{equation}
where $ H (K) $ is defined in \eqref{e:Hdef}.  Since $ m \geq C \delta ^{-2} H (K) ^2 $, and by 
 the general large deviation result for empirical processes Lemma~\ref{l:martingale}, we can then 
conclude our claim in high probability. In fact, appealing to concentration of measure again, we do not directly estimate the expectation  above, but rather 
show that a sufficiently small quantile of the r.v.\ is dominated by $ C  H (K)/\sqrt m$, hence the expectation is as well.

To bring the hemisphere gaussian process into play,  define 
\begin{equation}  \label{e:Z}
Z _{x,y}  := \frac 1 {\sqrt m} \sum_{k=1} ^{m}  \varepsilon  _{j}\mathbf 1_{W _{x,y}} (\theta _j)
\end{equation}
where $ \{\varepsilon _j\}$ is an iid sequence of rademacher random variables.  We will show that if $ m \gtrsim \delta ^{-2} H(K) $, 
\begin{equation}\label{e:2show}
\mathbb E \sup _{x,y\in K_0} \lvert  Z _{x,y} \rvert \lesssim H(K_0)  \leq H (K).   
\end{equation}
This is sufficient, since we can apply the symmetrization estimate \eqref{e:rad}, 
\begin{align} \label{e:Zz}
\mathbb E \sup _{x,y\in K} \Bigl\lvert  \frac 1 m \sum_{k=1} ^{m} \mathbf 1_{W _{x,y}} (\theta _j) - d (x,y) \Bigr\rvert  
\lesssim  \frac 1 {\sqrt m}\mathbb E \sup _{x,y\in K} \lvert  Z _{x,y}\rvert \lesssim H(K)/\sqrt m \lesssim \delta ,  
\end{align}
where the last quantity follows from our assumption on $ m$.  

\medskip 

Holding the $ \theta _j$ fixed, associated to the process $ Z _{x,y}$ is the conditional metric 
\begin{equation*}
D (x,y ) ^2  := \frac 1 m  \sum_{j=1} ^{m} \mathbf 1_{W _{x,y}}(\theta _j) . 
\end{equation*}
It remains to show that it is  very close to the metric for the hemisphere gaussian process.

%%%%%%%%%%%%%%%%%%%%%%%%%%%%%% LEMMA LEMMA LEMMA
\begin{lemma}\label{l:D} For any $ \epsilon >0$, there is  a constant $ C_ \epsilon $ sufficiently large, so that for all   $  m \geq C_ \epsilon  \delta ^{-2} H(K)$,  with probability at $ 1- \epsilon $, there holds 
\begin{equation*}
\sup _{x,y\in K_0}     \frac {\lvert  D (x,y) ^2 - \mathbb P (W _{x,y})\rvert } { \mathbb P (W _{x,y})} \leq  1 
\end{equation*}
\end{lemma}
%%%%%%%%%%%%%%%%%%%%%%%%%%%%%% LEMMA LEMMA LEMMA

Thus, on a set of large measure the conditional sub-gaussian process given by $ Z _{x,y}$ has a metric 
dominated by twice  that of the metric for the hemisphere gaussian process. Thus, for any $ 0<\epsilon <1 $, there is 
a $ C_ \epsilon >0$ so that  
\begin{equation*}
\mathbb P \bigl( \sup _{x,y \in K_0} \lvert  Z_x - Z_y\rvert \geq C_ \epsilon  H (K) \bigr) \leq 2 \epsilon . 
\end{equation*}
It follows from the deviation inequality \eqref{e:martingale} and standard concentration of measure estimates 
for gaussian processes that \eqref{e:2show} holds, completing the proof.

We now turn to the proof of Lemma \ref{l:D}.
%%%%%%%%%%%%%%%%%%%%%%%%%%%%%% PROOF PROOF PROOF
\begin{proof}
Fix $ x,y\in K_0$ and let $ p = \mathbb P  (W _{x,y}) = d(x,y)$. It is a consequence of Bernstein's inequality that 
\begin{align*}
\mathbb P ( \lvert  D (x,y) ^2 - \mathbb P (W _{x,y})\rvert  >  p ) 
& \leq C_0
\operatorname {exp}
\Bigl( - \frac { p ^2 } { p/m + 3 p/m } \Bigr) \leq 
C_0 \operatorname {exp}( - mp/4) . 
\end{align*}
By construction of $K_0$, $p = d (x,y) > \delta/4$. Hence  $ mp \gtrsim m \delta > C \delta ^{-1} H (K) \ge C' \log N (K, \delta /8)$, 
with the last inequality following from Sudakov's lower bound for the supremum of gaussian processes, see Theorem~\ref{t:sudakov} as well as the discussion of \eqref{e:compare}.
The constant $ C'$ can be made as large as desired, and there are at most $ N (K, \delta /8) ^2 $ pairs 
of $ x,y \in K_0$, so the proof is immediate from the union bound.  

\end{proof}
%%%%%%%%%%%%%%%%%%%%%%%%%%%%%% PROOF PROOF PROOF

%%%%%%%%%%%%%%%%%%%%%%%%%%%%%% SUBSECTION SUBSECTION SUBSECTION SUBSECTION
 %%%%%%%%%%%%%%%%%%%%%%%%%%%%%% SUBSECTION SUBSECTION SUBSECTION SUBSECTION 
\subsection{Remarks}%\label{ss.}
Plan and Vershynin \cite{MR3164174} proved analogs of the results above using the 
 gaussian process $ \gamma _x = \langle x, \gamma  \rangle$, which has a metric that differs from 
 that of the hemisphere process $ G_x$ by essentially a square root,  compare \eqref{e:g_metric} and \eqref{e:hemi}.  
Recall that the set $ K_0$ is a $ c\delta $ packing in $ K$, relative to the geodesic metric $ d (x ,y)$.  
 This allows a direct comparison of $ \omega (K_0)$ and $ H (K_0)$, namely 
$
H (K_0) \lesssim \delta ^{-1/2} \omega (K_0) 
$, 
which follows  from Talagrand's majorizing measure theorem.  
The key inequality  \eqref{e:2show}  proved about $ K_0$ shows that if $ m \gtrsim \delta ^{-3} \omega (K) ^2  \gtrsim \delta ^{-2} H (K_0) $,  
then  there is a $ \delta $-RIP map from $ K$ to the Hamming cube of dimension $ m$.  
And, when restricted to $ K_0$, this is the map $ x \mapsto \textup{sgn} (Ax)$, where the rows of $ A$ 
are the iid $ \theta _j$.  

This has the disadvantage of not making the RIP explicit on all of $ K$.  Plan-Vershynin \cite{MR3164174} show 
more, at the cost of  additional powers of $ \delta ^{-1} $.  
If $ m \gtrsim \delta ^{-6} \omega (K) ^2 $, then the map $ x \mapsto \textup{sgn} (Ax)$, \emph{on all of $ K$},  
is a $ \delta $-RIP, with high probability.  

It is reasonable to conjecture that the hemisphere constant $ H (K)$ is the correct quantity governing the 
RIP property into the Hamming cube. Namely, that 
if $ m \gtrsim \delta ^{-2} H (K) ^2 $, then with high probability, the map $ x \mapsto \textup{sgn} (Ax)$ 
is a $ \delta $-RIP map from $ K$ into the $ m$-dimensional Hamming cube.  
Verifying this conjecture seems to  
bump up against subtle questions about empirical processes indexed by hemispheres, and attendant issues 
related to concentration of measure in the sphere.  That is why we resort to a small ambiguity about exactly 
what the RIP is.  These issues do not arise in the sparse vector case, however, as is argued below. 

%%%%%%%%%%%%%%%%%%%%%%%%%%%%%% SUBSECTION SUBSECTION SUBSECTION SUBSECTION
 %%%%%%%%%%%%%%%%%%%%%%%%%%%%%% SUBSECTION SUBSECTION SUBSECTION SUBSECTION 
\subsection{The Case of Sparse Vectors in Theorem~\ref{t:H}}%\label{ss.}

We turn to    the  case of  sparse vectors, where we can give a   proof 
that the natural map $ x \mapsto \textup{sgn} (Ax)$ is a $ \delta $-isometry of the 
set  $ K= K _{s}$, for integer $ 0 < s \leq n$ into the Hamming cube. 
The very short proof is based on observations about VC classes and empirical processes 
that are collected  in the concluding section of the paper. 

Observe that for $ x,y\in K_s$, we have the empirical process identity 
\begin{equation*}
\lvert  d _{H} (x,y) - d (x,y)\rvert 
= 
\Bigl\lvert  \frac 1m \sum_{j=1} ^{m}   \mathbf 1_{W _{x,y}} (\theta _j) - d (x,y) \Bigr\rvert . 
\end{equation*}
Apply the empirical process inequality \eqref{e:W} with $ \eta =1$, and $ u= C\sqrt {s \log_+ \frac ns}$.  
We conclude that the bound below holds with 
probability at least $ 1 - \operatorname {exp}( - C s \log_+ \frac ns)$.  
\begin{equation*}
\sup_{x,y  \in K _{s}}  
\Bigl\lvert  \frac 1 {\sqrt m} 
\sum_{j=1} ^{m}  \big(  \mathbf 1_{W _{x,y}} (\theta _j) - d (x,y)  \big) \Bigr\rvert 
\lesssim \sqrt { s \log_+ \frac ns } 
\end{equation*}
With the  normalization by $ \sqrt m$ above,  our condition $ m \gtrsim \delta ^{-2} s \log_+ \frac ns $ clearly gives the desired conclusion.

%%%%%%%%%%%%%%%%%%%%%%%%%%%%%% SECTION  SECTION SECTION
%%%%%%%%%%%%%%%%%%%%%%%%%%%%%% SECTION  SECTION SECTION 
\section{Sign-Product Embedding Property} \label{s:sign}

We turn to the analysis of Theorem~\ref{t:sign_product}.  For a standard gaussian $ g$ on $ \mathbb R ^{n}$ 
and $ \lambda = \sqrt {\tfrac 2 \pi }$,  
we argue that   this inequality holds  for $ x, y \in \mathbb S ^{n}$. 
\begin{equation*}
\mathbb E\,  \textup{sgn} (\langle x,g \rangle) \langle y, g \rangle = \lambda  \langle x,y \rangle. 
\end{equation*}
Now, if $ \langle x,y \rangle=0$, then $ \langle x,g \rangle$ and $ \langle y,g \rangle$ are independent, 
verifying the equality above.  And, by linearity in expectation and  $ y$, we can then reduce to the case of $ x=y$. 
But then, the random variable above is the absolute value of $ Z$, a standard gaussian on $ \mathbb R $.  
And $ \mathbb E \lvert  Z\rvert= \lambda $.  
Thus, the bound in \eqref{e:sign_product} fits within the empirical process framework.

We will provide a proof that with probability at least $  1 - (\frac s {2n}) ^{2s}$, we have 
\begin{equation}  \label{e:SP<}
\sup _{x,y\in K_s} 
\biggl\lvert \frac 1 {\sqrt m }\sum_{j=1} ^{m}  \textup{sgn} (\langle x, g_j \rangle) \langle y, g_j \rangle - \lambda 
\langle x,y \rangle\biggr\rvert 
\leq C \sqrt {s  \log_+ \frac ns  }  ,
\end{equation}
where $  K_s$ is the collection of $ s$-sparse  vectors in $ \mathbb S ^{n}$.     
For each fixed $ x, y\in K_s$, we certainly have 
\begin{equation*}
\mathbb P \Bigl(\biggl\lvert \frac 1 {\sqrt m }\sum_{j=1} ^{m}  \textup{sgn} (\langle x, g_j \rangle) \langle y, g_j \rangle - \lambda 
\langle x,y \rangle\biggr\rvert > 2  \Bigr) \leq \frac 12 . 
\end{equation*}
It follows that the symmetrization inequality Lemma~\ref{l:anderson} holds.  Namely is suffices to show that 
with probability at least $  1 -\tfrac 14 (\frac s {2n}) ^{2s}$, we have 
\begin{equation}\label{e:SPP<}
\sup _{x,y\in K_s} 
\sup _{x,y\in K_s} 
\biggl\lvert \underbrace {\frac 1 {\sqrt m } \sum_{j=1} ^{m}  \varepsilon _j\textup{sgn} (\langle x, g_j \rangle) \langle y, g_j \rangle} _{:= Z (x,y)} \biggr\rvert
\leq C \sqrt {s  \log_+ \frac ns  }  . 
\end{equation}
Above, we take $ \varepsilon _j$ to be an independent set of rademacher r.v.s.

To ease notations, call the sum above $ Z(x,y)$,   set $ \overline x_j = \textup{sgn} (\langle x, g_j \rangle) $, 
and $ \overline y_j = \langle y,g_j \rangle$.  
It is essential to observe that $ \varepsilon_j \overline x_j \overline y_j$
is distributed like a one dimensional mean zero gaussian, of variance $ \lVert y\rVert ^2 $. 
(On the other hand $ \varepsilon_j (\overline x_j - \overline  x '_j)\overline y_j$ is not gaussian, since it will 
 equal zero if $ g_j/\lVert g_j\rVert \not\in W _{x, x'}$.)

The proof of \eqref{e:SPP<} then combines three ingredients.  
(1)  For fixed $ x\in K_s$, the r.v. $ \sup _{y\in K_s} \lvert  Z (x,y)\rvert $ 
is controlled, and has concentration of measure. We can, essentially for free, form a supremum over 
$ x$ in net $ X\subset K_s$ of small diameter.  
(2) As $ x', x'' \in K_s$ vary over all pairs of vectors that are close, the selector random variables 
$  \overline  x'_j -\overline  x''_j \in\{-1, 0, 1\} $ pick out a set of small cardinality.  
And (3) the supremum over the sum of the $ \overline y_j$ over sets of small cardinality is controlled.  

\smallskip 

The details are as follows.  
For fixed $ x$,  the process $ \{Z (x,y) \;:\; y\in K_s\}$ is a gaussian process, with mean 
given in terms of  the gaussian mean width of $ K_s$. 
\begin{equation*}
\mathbb E \sup _{y\in K_s} 
 \lvert  Z (x,y)\rvert  \lesssim \omega (K_s)  \lesssim \sqrt {s \log_+ \frac ns } .  
\end{equation*}
Furthermore, gaussian processes have a concentration of measure around their means. 
Therefore, for any finite set $X\subset  K_s $, provided 
\begin{equation}\label{e:X<}
\log {} ^{\sharp} X \lesssim s \log_+ \frac ns , 
\end{equation}
we have, with probability at least $  1 - (\frac s {10n}) ^{2s}$, 
\begin{equation}\label{e:X}
\sup _{x\in X, y\in K_s} 
 \lvert  Z (x,y)\rvert \lesssim  \sqrt {s \log_+ \frac ns }. 
\end{equation}

Take $ X \subset K_s$ to be a minimal cardinality collection of 
vectors so that for all $ y\in K_s$, there is an $ x\in X$ with $ d (x,y) \leq \eta =   (\frac s{100n}) ^6  $.
Observe that by the VC property of wedges and Lemma~\ref{l:VC-entropy}, 
\begin{equation*}
{} ^{\sharp} X \lesssim \binom {n+1} s \eta ^{-3 (s+1)}  \lesssim (100 \frac ns) ^{ 20(s+1)}. 
\end{equation*}
This is the estimate \eqref{e:NW}, applied with the uniform probability measure on the sphere. 
Therefore, \eqref{e:X<} holds.  This completes the first stage of the argument. 

\smallskip 
In the second stage, we  have the differences  $ Z (x,y) - Z (x',y)$, which involve  
 $ \overline  x_j - \overline x'_j \in \{-1, 0, 1\}$.   
But these last differences will be non-zero on a small set of indices. 
Indeed, with probability at least $   1 - (\frac s {10n}) ^{2s}$, the inequality below holds.  
\begin{equation} \label{e:small}
\sup _{ \substack{x,x'\in K_s\\ d (x,x') < \eta   }} 
 \sum_{j=1} ^{m}  \lvert  \overline x_j -\overline x_j'\rvert  \leq C \sqrt { m  \eta s \log_+  \frac ns } = u.    
\end{equation}

The left side of \eqref{e:small} is the 
un-normalized Hamming metric between the 
natural mapping of $ x$ and $ x'$ in the Hamming metric. 
With our choice of $ \eta = (\frac s{100n}) ^6 $, note that 
\begin{equation*}
\eta \leq  \sqrt { m  \eta s \log_+  \frac ns }. 
\end{equation*}
Subtracting the actual Hamming metric,      we should bound the supreumum 
\begin{equation}  \label{e:Y}
\sup _{ \substack{x,x'\in K_s\\ d (x,x') < \eta   }} 
\Bigl\lvert  \sum_{j=1} ^{m} \mathbf 1_{W _{x,x'}} (\theta _j) - d (x,x') \Bigr\rvert
\end{equation}
where $ \theta _j = g_j / \lVert g_j\rVert$.    Note that this is a empirical process for which we have 
the concentration of measure of Lemma~\ref{l:martingale}.  Thus, the high probability part of our 
assertion will follow from this fact.  

Now,  Theorem~\ref{t:W}  proves that the expectation in \eqref{e:Y} is bounded by 
\begin{align*}
\sqrt m   \int _{0 } ^{\eta ^{1/2} } \sqrt {s\log _+ \tfrac ns  + s \log_+ \tfrac 1t  }  \; dt 
  \lesssim    \sqrt { m \eta s \log_+ \frac ns} 
\end{align*}
So, \eqref{e:small} holds, and the second stage of the argument is complete. 

\smallskip 

The third and final stage is to show that with probability 
$  1 - (\frac s {10n}) ^{2s}$, 
\begin{equation}\label{e:SupJ}
 \sup _{\substack{J \subset \{1 ,\dotsc, m\}\\ \lvert  J\rvert  \leq u } } \sup _{y\in K_s}   \biggl\lvert \frac 1 {\sqrt m} \sum_{j \in J} \varepsilon _j \overline y_j \biggr\rvert \lesssim \sqrt {s \log_+ \frac n s}.  
\end{equation}
Above, $ u$ is as in \eqref{e:small}.  

The sum in \eqref{e:SupJ} is that of a gaussian process, hence to deduce the claim, it suffices to bound the 
expectation on the left in \eqref{e:SupJ}. 
For fixed $ J \subset \{1 ,\dotsc, m\}$, of cardinality at most $ u$,  the supremum below  
\begin{equation*}
 \sup _{y\in K_s}   \biggl\lvert \frac 1 {\sqrt m} \sum_{j \in J} \varepsilon _j \overline y_j\biggr\rvert
\end{equation*}
is that  of a gaussian process with metric dominated by $  \sqrt {u/m} \lVert  y-y'\rVert_2 $.  
Therefore, the estimate below is immediate. 
\begin{equation*}
\mathbb E  \sup _{y\in K_s}   \biggl\lvert \frac 1 {\sqrt m} \sum_{j \in J} \varepsilon _j \overline  y_j\biggr\rvert 
\lesssim  \sqrt {u/m}   \omega (K_s) \lesssim \sqrt { \frac um \cdot    s \log_+ \frac ns   }.   
\end{equation*}
Since this is universal in $ J$, the next step is to appeal to concentration of measure above the mean. 
For this, we just need to count the number of sets $ J $ of cardinality at most $ u$, of which there are 
clearly at most $ 2\binom mu  $.   
We conclude that 
\begin{equation}   \label{e:ZX}
E \sup _{\substack{J \subset \{1 ,\dotsc, m\}\\ \lvert  J\rvert  \leq u } } \sup _{y\in K_s}   \biggl\lvert \frac 1 {\sqrt m} \sum_{j \in J} \varepsilon _j \overline y_j \biggr\rvert 
\lesssim 
\sqrt { \frac u m} \Bigl( \sqrt { s \log_+ \frac ns }+    \sqrt{u    \log_+ \frac m u }  \Bigr) . 
\end{equation}

It remains to see that the right side of \eqref{e:ZX} is at most $ C \sqrt {s \log_+ \frac ns}$, 
and from this fact, the conclusion \eqref{e:SupJ} follows from concentration of measure.  
The right side of \eqref{e:ZX} has two terms, the first is obviously less than $ \sqrt {s \log_+ \frac ns}$, 
since $  u \leq  m$.  The second is  
\begin{align*}
\frac {  u} {\sqrt m} \sqrt {\log_+ \frac m { u}} 
&=    \sqrt { \eta s  \log_+   \frac  ns }    \times \sqrt { \log_+ \frac m  {\eta  s \log_+  \frac ns}}  
\\
& \lesssim   \sqrt {s \log_+ \frac ns}  \Bigl(\frac s {100 n} \Bigr) ^{3} 
 \sqrt { \log_+ \frac m  {  s \log_+  \frac ns}}  
\\
& \lesssim 
 \sqrt {s \log_+ \frac ns}  \bigl(\frac s {100 n} \bigr) ^{3} \sqrt { \log_+  \delta ^{-2}}  
  \lesssim  \sqrt {s \log_+ \frac ns}  . 
 \end{align*}
Here, we have used the choice of $ \eta =  \bigl(\frac s {100 n} \bigr) ^{6} $, 
$ m \approx \delta ^{-2} s \log_+ \frac ns$,  and the definition 
of $  u$ in \eqref{e:small}.  The last inequality 
follows from the condition on $ \delta $ in \eqref{e:assume}.  This completes the proof.

%%%%%%%%%%%%%%%%%%%%%%%%%%%%%% SECTION  SECTION SECTION
%%%%%%%%%%%%%%%%%%%%%%%%%%%%%% SECTION  SECTION SECTION 
\section{Background on Stochastic Processes} \label{s:prob}

%%%%%%%%%%%%%%%%%%%%%%%%%%%%%% SUBSECTION SUBSECTION SUBSECTION SUBSECTION
 %%%%%%%%%%%%%%%%%%%%%%%%%%%%%% SUBSECTION SUBSECTION SUBSECTION SUBSECTION 
\subsection{Gaussian Processes}%\label{ss.}

A process $ \{Z_t \;:\; t\in T\}$ is a mean zero gaussian process iff for every finite subset $ T'\subset T$, 
the restriction $ \{Z_{t'} \;:\; t\in T\}$ is a finite dimensional mean zero  gaussian vector.  
The process $ Z_t$ induces a metric on $ T$ by 
\begin{equation*}
d _{Z} (s,t) = \lVert X_s - X _t\rVert_2  = \big( \mathbb E \, |X_t - X_s|^2 \big)^{1/2} . 
\end{equation*}
%The important quantity for us 
We shall be interested in the bounded sample path properties, which are  measured by 
\begin{equation*}
\mathbb E \sup _{t} \lvert  Z_t\rvert, \qquad  \mathbb E \sup _{s,t} \lvert  Z_s -Z_t\rvert. 
\end{equation*}
The second expression can be written without absolute values. 

Since a gaussian process is uniquely defined by its mean and covariance, any subset $ K$ of a Hilbert space $ H$ generates  a mean zero gaussian process $ \{Z_t \;:\; t\in K\}$, where 
the metric $ d_Z$ is given by the Hilbert space metric, $ d _{Z} (s,t) = \lVert s-t\rVert_ H $.  
This is in fact how the gaussian mean width and hemisphere processes could be defined. 
If the set $ K$ is symmetric, then the suprema of the gaussian processes, with or without the 
absolute values, are the same.  

The quantity $ \mathbb E \sup _{t\in T} Z_t$ is of basic interest for us.  A foundational  result concerning this quantity is 
the Sudakov lower bound. 

%%%%%%%%%%%%%%%%%%%%%%%%%%%%%% THEOREM THEOREM THEOREM
\begin{theorem}\label{t:sudakov}[Sudakov Lower Bound]  
For a gaussian process $ \{Z_t \;:\; t\in T\}$, one has 
\begin{equation*}
\sup _{\delta } \delta \sqrt {\log N (T, d_Z, \delta )} \lesssim \mathbb E \sup _{s,t} \lvert  Z_s -Z_t\rvert, 
\end{equation*}
where $ N (T, d_Z , \delta )$ is the covering number of $T$ with respect to the metric of the gaussian process $Z$, i.e.  the least number of $ d_Z$-balls of radius $ \delta $ needed to cover $ T$. 
\end{theorem}
%%%%%%%%%%%%%%%%%%%%%%%%%%%%%% THEOREM THEOREM THEOREM

Characterizing the bounded sample path properties of gaussian processes is an important  
accomplishment of Talagrand in the majorizing measure theorem \cite{MR3184689}, to which we referred in the text. \\
%We refer to this theorem at one point below. 

%
%
%%%%%%%%%%%%%%%%%%%%%%%%%%%%%%%  DEFINITION DEFINITION DEFINITION
%\begin{definition}\label{d:gamma}  The quantity $ \gamma _2 (T, d)$ is defined as follows.  
%For integers $ s$  with 
%\begin{equation*}
%s > s_0 = - \lceil \log \sup _{s,t\in T} d (s,t)\rceil , 
%\end{equation*}
%consider  sets $ K_s \subset K$ with $ \lvert  K_s\rvert \leq 2 ^{2 ^{s}} $, 
%and maps $ \pi _s (x) := \textup{argmin} \{  \lvert  x-y\rvert_2 \;:\; y\in K_s \}$ from $ K$ to $ K_s$. 
%Define 
%\begin{equation}\label{e:gamma2}
%\gamma _2 (T,d) = \inf _{ \{K_s\}}\sup _{x\in T} \sum_{s=0} ^{\infty } 2 ^{s/2} d (x, \pi _{s} (x)) \lesssim \omega (K). 
%\end{equation}
%
%\end{definition}
%%%%%%%%%%%%%%%%%%%%%%%%%%%%%%%  DEFINITION DEFINITION DEFINITION
%
%%%%%%%%%%%%%%%%%%%%%%%%%%%%%%% THEOREM THEOREM THEOREM
%\begin{theorem}\label{t:MMT}[Talagrand's Generic Chaining Theorem]  
%For any gaussian process, there holds 
%\begin{equation}\label{e:mmt}
% \mathbb E \sup _{s,t} \lvert  Z_s -Z_t\rvert \simeq \gamma _2 (T, d_Z). 
%\end{equation}
%\end{theorem}
%%%%%%%%%%%%%%%%%%%%%%%%%%%%%%% THEOREM THEOREM THEOREM

We also make use of the concentration of measure. % is important.  
%%%%%%%%%%%%%%%%%%%%%%%%%%%%%% THEOREM THEOREM THEOREM
\begin{theorem}\label{t:GC}[Gaussian Concentration of Measure] 
Let $ Z_ t$ be a mean zero Gaussian process with $ \sigma ^2 = \sup _{t\in T} \mathbb E Z_t ^2  $ the maximal 
variance.  There is  a constant $ c>0$ so that for all $ \lambda >0$
\begin{equation}\label{e:GC}
\mathbb P (c \sup _{t} Z_t > \mathbb E \sup _{t\in T} Z_t +   \lambda ) 
\lesssim \operatorname {exp}(  - (\lambda / \sigma ) ^2 ).  
\end{equation}
\end{theorem}
%%%%%%%%%%%%%%%%%%%%%%%%%%%%%% THEOREM THEOREM THEOREM

A random variable $ X$ is said to be \emph{sub-gaussian with  parameter $ \sigma $} if 
\begin{equation}\label{e:sub}
\mathbb P (\lvert  X\rvert > \lambda  ) \lesssim \operatorname {exp}(  - (\lambda / \sigma ) ^2 ), \qquad \lambda >0. 
\end{equation}
The supremum of a finite number of  sub-gaussian random variables grows very slowly. 

%%%%%%%%%%%%%%%%%%%%%%%%%%%%%% LEMMA LEMMA LEMMA
\begin{lemma}\label{l:sup} Let $ X_1 ,\dotsc, X_N$ be sub-gaussian r.v.s with maximal parameter $ \sigma $. 
Then, 
\begin{equation}\label{e:sup}
\mathbb E \sup _{1\le n \leq N} \lvert  X_n\rvert \lesssim \sqrt {\log N}.   
\end{equation}

\end{lemma}
%%%%%%%%%%%%%%%%%%%%%%%%%%%%%% LEMMA LEMMA LEMMA

%%%%%%%%%%%%%%%%%%%%%%%%%%%%%% SUBSECTION SUBSECTION SUBSECTION SUBSECTION
 %%%%%%%%%%%%%%%%%%%%%%%%%%%%%% SUBSECTION SUBSECTION SUBSECTION SUBSECTION 
\subsection{Empirical Processes}%\label{ss.}

Let $ (T, \mu )$ be a probability space, and $ X_1 ,\dotsc, X_m$ are iid random variables, taking values in $ T$ 
with distribution $ \mu $.  For a class of functions $ f \in \mathcal F $, mapping $ T$ into $ \mathbb R $, we define 
the empirical process 
\begin{equation*}
S_m f =  \frac 1 m \sum_{j=1} ^{m} f (X_j) - \int _{T} f \; d \mu ,  \qquad f\in \mathcal F. 
\end{equation*}
Above, we must assume that $ f\in L ^{1} (T, \mu )$, and while 
the case of general functions is quite interesting, but for our purposes we can always assume that $ f$ are 
bounded functions.  Indeed, we mostly work with the case $\mathcal F = \{ {\bf{1}}_{W_{x,y}} \}$. One is then interested in variants of the law of large numbers and central limit theorem 
in this context.   The quantity important to us is 
\begin{equation*}
S _{m} (\mathcal F) :=   \sup _{f\in \mathcal F} \lvert  S_m f\rvert  .  
\end{equation*}
Basic to the analysis of these processes is   \emph{symmetrization},  which we state in two forms.  

%%%%%%%%%%%%%%%%%%%%%%%%%%%%%% LEMMA LEMMA LEMMA
\begin{lemma}\label{l:anderson} \cite{MR757769}
Fix $  s >0$ so that 
\begin{equation*}
\sup _{f\in \mathcal F} \mathbb P (\lvert  S_m f \rvert > s) \leq \tfrac 12 . 
\end{equation*}
Then, for $ t>0$, there holds 
\begin{equation*}
\mathbb P \bigl( \sup _{f\in \mathcal F} \lvert   S_m f \rvert >t   \bigr) 
\leq 2 \mathbb P \bigl( \sup _{f\in \mathcal F} \lvert   S_m f - S'_m f \rvert >t   -s  \bigr) 
\end{equation*}
where $ S'_m f$ is an independent copy of $ S_m f$. 
\end{lemma}
%%%%%%%%%%%%%%%%%%%%%%%%%%%%%% LEMMA LEMMA LEMMA

A second symmetrization lemma for expectation is as follows. 

%%%%%%%%%%%%%%%%%%%%%%%%%%%%%% LEMMA LEMMA LEMMA
\begin{lemma}\label{l:symmetry}[Symmetrization]  
For any class of $ \mu $-integrable functions $ \mathcal F$ one has 
\begin{align}\label{e:rad}
\mathbb E S _m (\mathcal F)
&\leq 2\mathbb E \sup _{f\in \mathcal F}  
\Bigl\lvert  \frac 1 m \sum_{j=1} ^{m}  r_j f (X_j) \Bigr\rvert 
\\  \label{e:symmetry}
&\leq 2
\mathbb E \sup _{f\in \mathcal F}  
\Bigl\lvert  \frac 1 m \sum_{j=1} ^{m}  g_j f (X_j)  \Bigr\rvert . 
\end{align}
Above, $ \{r_j\}$ denote Rademacher and $ \{g_j\}$ standard gaussian random variables, both independent of the $ \{X_j\}$. 
\end{lemma}
%%%%%%%%%%%%%%%%%%%%%%%%%%%%%% LEMMA LEMMA LEMMA

In the first line \eqref{e:rad}, conditional on the $ \{X_j\}$, the process is a \emph{rademacher process}, 
which is subgaussian, in the sense of \eqref{e:sub}. Hence, it is dominated by the conditional 
gaussian process in \eqref{e:symmetry}.  
The conditional  metric  on the set $ \mathcal F$ is of fundamental importance. It is 
\begin{equation} \label{e:randomMetric}
d (f,g) =  \frac 1 m \sum_{j=1} ^{m}   \lvert f (X_j) - g (X_j)\rvert ^2 .  
\end{equation}

Estimating the term $ S_m (\mathcal F)$ is fundamental in the case when $ \mathcal F$ consists 
of functions bounded by one, in view of the following  deviation inequality, which is an application of the Hoeffding inequality, 
and is sometimes called the McDiarmid inequality. 

%%%%%%%%%%%%%%%%%%%%%%%%%%%%%% LEMMA LEMMA LEMMA
\begin{lemma}\label{l:martingale}\cite{MR882849} Suppose that $ \mathcal F$ consist of functions bounded by one.  
There holds 
\begin{equation}\label{e:martingale} 
\mathbb P (  S_m (\mathcal F)  > \lambda +  \mathbb E  S_m (\mathcal F) ) 
\lesssim \operatorname {exp}(- m \lambda ^2 /2 )
\end{equation}

\end{lemma}
%%%%%%%%%%%%%%%%%%%%%%%%%%%%%% LEMMA LEMMA LEMMA

%%%%%%%%%%%%%%%%%%%%%%%%%%%%%%% PROOF PROOF PROOF
%\begin{proof}
%Notice that 
%\begin{equation*}
%\phi _ j =  \mathbb E \bigl( S_m (\mathcal F) \;|\;   X_1 ,\dotsc, X_j \bigr) 
%\end{equation*}
%is a martingale, with $ \phi _{m} = S_m (\mathcal F)$.   Moreover, the increments are bounded: 
%\begin{align*}
%\lvert    \phi _{j+1} - \phi _j \rvert  & = 
%\bigl\lvert  \mathbb E \bigl( \mathbb E  (S_m (\mathcal F) \;|\;  X_ {j+1} )  
%- S_m (\mathcal F) \;|\; X_1 ,\dotsc, X_j \bigr) 
%\bigr\rvert
%\leq \frac 1m
%\end{align*}
%since only the value of $ X _{j+1}$ can influence the difference, and it's contribution to the emprical measure is 
%at most $ 1/m$.  But, it follows that the square function of the martingale $ \phi _j$  is
%\begin{equation*}
%S (\phi ) ^2 := \sum_{j=1} ^{m}  \lvert  \phi _j - \phi  _{j-1}\rvert ^2 \leq \frac 1 m.   
%\end{equation*}
%And, then the inequality of the Lemma is an instance of the Azuma-Hoeffding inequality.  
%\end{proof}
%%%%%%%%%%%%%%%%%%%%%%%%%%%%%%% PROOF PROOF PROOF

To estimate $ S_m (\mathcal F)$, where $ \mathcal F $ consists of indicator sets, one should 
divide $ \mathcal F$ into parts which are either governed by the gaussian theory, or have  no cancellative part, and so should be controlled by Poisson like behavior.  
While there are several techniques here, one should not forget that concentration of measure 
on the sphere is an obstacle to the use of techniques such as `bracketing.'  
A succinct summation of the main conjectures in the subject are in \cite{MR3184689}*{Chap. 9}. 

%%%%%%%%%%%%%%%%%%%%%%%%%%%%%% SUBSECTION SUBSECTION SUBSECTION SUBSECTION
 %%%%%%%%%%%%%%%%%%%%%%%%%%%%%% SUBSECTION SUBSECTION SUBSECTION SUBSECTION 
\subsection{VC Dimension}%\label{ss.}

Let $ (M, \mathcal M) $ be a measure space, and $ \mathcal C \subset M$ a class of measurable sets.  
For integers $ n$, let 
\begin{equation*}
S (n, \mathcal C) := \sup _{ \substack{A\subset M\\  \lvert  A\rvert=n  }}  \lvert   \{ B \;:\; B= A \cap C,\ C\in \mathcal C\}\rvert 
\end{equation*}
That is, $ S (n)$ is the largest number of subsets that can be formed by intersecting a set $ A$ of cardinality 
$ n$ with sets $ C\in \mathcal C$.  
It is clear that $ S (n) \leq 2 ^{n}$. 
 The Vapnik-Cervonenkis dimension  (VC-dimension) 
  $\nu (\mathcal C)$ of $ \mathcal C$ is the least integer $ d$ such that  $ S (n) < 2 ^{n}$ for all $ n>d$.  
The Sauer Lemma then states that 
\begin{equation}\label{e:shelah}
S (n) < \Bigl(\frac {n  e ^{}} {d}  \Bigr) ^{d},  \qquad n > 1.  
\end{equation}

For us, the class of relevant sets are \emph{hemispheres.}  More generally, for spherical caps in $ \mathbb S ^{n}$, 
the VC-dimension is      $n+1 $. 

%%%%%%%%%%%%%%%%%%%%%%%%%%%%%% PROPOSITION PROPOSITION PROPOSITION
\begin{proposition}\label{p:vch} The family   $ \mathcal C_{n}$ of spherical caps on $ \mathbb S ^{n}$ has VC-dimension 
 $ n+1$.  
\end{proposition}
%%%%%%%%%%%%%%%%%%%%%%%%%%%%%% PROPOSITION PROPOSITION PROPOSITION

%%%%%%%%%%%%%%%%%%%%%%%%%%%%%% PROOF PROOF PROOF
\begin{proof}
This is a modification of the proof in \cite{MR3000572}*{Prop. 8}. 
A collection  $ A$ of points on the unit sphere that is shattered by spherical caps has the additional 
property that for any partition of $ A $ into disjoint subsets $ A'$ and $ A''$, the convex hulls of $ A'$ 
and $ A''$ can't intersect.  This is just because a spherical cap is itself the intersection of the sphere with 
a half-space.  
But Radon's Theorem, a basic result in convex geometry,
says that any $ n+2$ points in $ \mathbb R ^{n}$ can be partitioned into two 
disjoint sets whose convex hulls intersect.   Hence $ \nu (\mathcal C_n) < n+2$.  

We show that spherical caps shatter the collection of  $ n+1$ vectors 
\begin{equation*}
A = \{\mathbf e_1 ,\dotsc, \mathbf e_ {n},\  \mathbf 1/\sqrt {n}\}  
\end{equation*}
where the $ \mathbf e_j $ is the standard basis elements, and $ \mathbf 1$ is the vector of all ones.  
Consider a partition of $ A$ into $ A'$ and $ A''$, supposing that $  \mathbf 1/\sqrt {n+1} \in A''$. 
If either $ A'$ or $ A''$ is a singleton, a small radius spherical cap around the singleton 
provides this decomposition.  Otherwise, observe that  there is a $ \theta \in \mathbb S ^{n}$ with 
\begin{equation*}
\langle x, \theta  \rangle =0 \quad x\in A', 
\end{equation*}
and  $ \langle x, \theta  \rangle >0 $ for $ x\in A''$.  Indeed, the first condition  
only requires that the $ j$th coordinates of $ \theta $ be zero, for $ \mathbf e_j \in A'$.  
And otherwise we require that $ j$th coordinate of $ \theta $ be positive, which we can 
do since $ A''$ is not a singleton.  
\end{proof}
%%%%%%%%%%%%%%%%%%%%%%%%%%%%%% PROOF PROOF PROOF

Thus, the class of hemispheres in $ \mathbb S ^{n}$ has dimension at most $ n+1$. 
Now, we are also interested in the class of symmetric differences of hemispheres. 
But, it follows from   estimate \eqref{e:shelah} that forming the the symmetric difference 
of a class $ \mathcal C$ of finite VC-dimension increases the dimension by at most a factor of 3.  
In symbols, $ \nu (\mathcal C \triangle  \mathcal C) \leq 3 \nu (\mathcal C)$.

The next important property of VC-classes is that they are universal for empirical processes. 
This is encoded in an inequality of this form.  

%%%%%%%%%%%%%%%%%%%%%%%%%%%%%% LEMMA LEMMA LEMMA
\begin{lemma}\label{l:VC-entropy}  Let $ \mathcal C$ be a collection of sets on a measure space $ (M, \mathcal M)$ 
with VC-dimension $ d$. Let $ P$ be any probability measure on $ (M, \mathcal M)$, and set 
metric $ d_P $ on $ \mathcal C$ by $ d_P (C_1, C_2) = P (C_1 \triangle C_2)$.  
For the metric entropy numbers $ N (\mathcal C, d_P, \delta )$, we have the estimate 
\begin{equation*}
N (\mathcal C, d_P, \delta ) \lesssim (\delta/2) ^{-4d}, \qquad 0< \delta < 1. 
\end{equation*}

\end{lemma}
%%%%%%%%%%%%%%%%%%%%%%%%%%%%%% LEMMA LEMMA LEMMA

%%%%%%%%%%%%%%%%%%%%%%%%%%%%%% SUBSECTION SUBSECTION SUBSECTION SUBSECTION
 %%%%%%%%%%%%%%%%%%%%%%%%%%%%%% SUBSECTION SUBSECTION SUBSECTION SUBSECTION 
\subsection{Uniform Entropy Classes}%\label{ss.}

It is convenient to recall a consequence of Lemma~\ref{l:VC-entropy} in somewhat greater  generality.  
The class  $ \mathcal W _{s} :=   \{ W_{x,y} :\; x,y \in K_s\}$ of symmetric differences of 
hemispheres associated to sparse vectors is not VC, unless we 
further restrict the coordinates in which $ x$ is not zero.  But, this class of sets is still `universal' in this sense.  
For any probability measure $ P$ on $ \mathbb S ^{n}$, define the metric $ d_P$ as in Lemma~\ref{l:VC-entropy}. 
We have the following  estimate on the metric entropy of $ \mathcal W _{s}$. 
\begin{equation}\label{e:NW}
N (\mathcal W_s, d_P, \delta ) \lesssim  {\binom {n+1}s} ^2  (\delta/2) ^{-12s-24}, \qquad 0< \delta < 1. 
\end{equation}
This is a direct consequence of Proposition~\ref{p:vch} and  Lemma~\ref{l:VC-entropy}, after holding the coordinates in which $ x$ and $ y$ 
are supported fixed.  

The class $ \mathcal W_s$ satisfies the \emph{uniform entropy bounds} of Panchenko. 
And, as a corollary to \cite{MR1887174}*{Cor. 1}, we derive the result below.  (See the beginning of \cite{MR1887174}* {\S1.1} 
        for relevant definitions. The proof of the Corollary is a standard, but somewhat delicate, chaining argument, using the uniform entropy bounds.)   

%%%%%%%%%%%%%%%%%%%%%%%%%%%%%% THEOREM THEOREM THEOREM
\begin{theorem}\label{t:W}  For $ \{ \theta _j\}$ iid and uniform on $ \mathbb S ^{n}$,  and $ 0 < \eta < 1$, 
and $ u> 0$, with probability at least $ 1 - 2 e ^{- u^2 }$, 
\begin{equation}\label{e:W}
\sup _{ \substack{x,y \in K_s\\ d (x,y) < \eta  }} 
\frac 1 {\sqrt m} \sum_{j=1} ^{m} \big( \mathbf 1_{W _{x,y}} (\theta _j) - d (x,y)  \big) 
\lesssim   \int _{0 } ^{\eta ^{1/2} } \sqrt {s\log _+ \tfrac ns  + s \log_+ \tfrac 1t  }  \; dt +   u \sqrt \eta .  
\end{equation}

\end{theorem}
%%%%%%%%%%%%%%%%%%%%%%%%%%%%%% THEOREM THEOREM THEOREM

%%%%%%%%%%%%%%%%%%%%%%%%%%%%%% PROOF PROOF PROOF
\begin{proof}
 In the language of Panchenko \cite{MR1887174}, the collection of wedges $ \{W _{x,y} \;:\; x,y\in K_s\}$ 
 satisfy \emph{uniform entropy bounds}, namely the inequality \eqref{e:NW} holds. 
 And the statement above is then a particular instance of \cite{MR1887174}*{Corollary 1}. 
\end{proof}
%%%%%%%%%%%%%%%%%%%%%%%%%%%%%% PROOF PROOF PROOF

%%%%%%%%%%%%%%%%%%%%%%%%%%%%%%% SUBSECTION SUBSECTION SUBSECTION SUBSECTION
% %%%%%%%%%%%%%%%%%%%%%%%%%%%%%% SUBSECTION SUBSECTION SUBSECTION SUBSECTION 
%\subsection{Measure Concentration on the Sphere}%\label{ss.}
%
%An obstruction to some proof techniques in one bit sensing arises from the measure concentration phenomena 
%on the sphere.  
%
%%%%%%%%%%%%%%%%%%%%%%%%%%%%%%% THEOREM THEOREM THEOREM
%\begin{theorem}\label{t:sphere} Let $ A \subset \mathbb S ^{n}$ be a set of measure $ 1/2$, then for $ 0 < \delta $, there holds 
%\begin{equation*}
%\sigma (A _{\delta }) \ge 1 - 2 \operatorname {exp}(- (n+1) \delta ^2 /2)
%\end{equation*}
%where $ A _{\delta } := \{z\in \mathbb S ^{n} \;:\;  \inf _{x\in A} \lvert  x-z\rvert\leq \delta   \}$ is a $ \delta $-neighborhood of $ A$. 
%\end{theorem}
%%%%%%%%%%%%%%%%%%%%%%%%%%%%%%% THEOREM THEOREM THEOREM
%
%As a corollary, if we take a $ \delta $-neighborhood of a wedge $ W _{x,y}$, this set will have very large measure 
%if $ \delta \gtrsim 1/ \sqrt n$.  

\end{document}